\documentclass[a4paper,12pt]{amsart} 
\usepackage{array}
\usepackage{enumitem}
\usepackage[colorlinks,linkcolor=blue]{hyperref} 
\usepackage{doi}

\DeclareMathOperator{\Spec}{Spec}
\DeclareMathOperator{\Proj}{Proj}

\DeclareMathOperator{\Hom}{Hom}

\DeclareMathOperator{\Ext}{Ext}

\DeclareMathOperator{\ev}{ev}

\def\mapleft#1{\mathrel{%
\smash{\mathop{\longleftarrow}\limits^{#1}}}}

\renewcommand{\P}{\mathbb{P}}

\newcommand{\N}{\mathbb{N}}
\newcommand{\A}{\mathbb{A}}
\newcommand{\ep}{\varepsilon}
\long\def\comment#1\endcomment{}

\newtheorem{thm}{Theorem}[section]

\newtheorem{lemma}[thm]{Lemma}

\newtheorem{prop}[thm]{Proposition}
\theoremstyle{definition}

\theoremstyle{remark}
\newtheorem{remark}[thm]{Remark}

\newcommand{\NN}{\Gamma} 
\newcommand{\M}{{\mathcal{M}_{g,1}^{\NN}}} 
\newcommand{\LO}{\mathcal{O}} 
\renewcommand{\k}{\mathbf{k}} 
\renewcommand{\l}{\ell} 

\def\y(#1){y_{#1}}
\def\a(#1){a_{#1}}
\def\b(#1){b_{#1}}

\def\f(#1)(#2){f_{#1}^{\scriptscriptstyle(#2)}}
\def\g(#1)(#2){g_{#1}^{\scriptscriptstyle(#2)}}
\def\ff(#1)(#2){f_{#1,#2}}
\def\gg(#1)(#2){g_{#1,#2}}

\newcommand{\half}{{\textstyle \frac12}}

\begin{document}
\title[Moduli of Low Genus Pointed Curves]{The Dimension of  the Moduli Space of Pointed Algebraic Curves of Low Genus}

\author{Jan Stevens}
\address{\scriptsize Department of Mathematical Sciences, Chalmers University of
Technology and University of Gothenburg.
SE 412 96 Gothenburg, Sweden}
\email{stevens@chalmers.se}

\begin{abstract}
We   explicitly  compute  the moduli space pointed algebraic curves  with a given numerical
semigroup  as Weierstrass semigroup for many cases of genus at most 
seven and determine the dimension  for all semigroups of genus seven. 
\end{abstract}

\thanks{}
\maketitle

\section*{Introduction}
On a smooth projective curve $C$ the pole orders of rational functions
with poles only at a given point $P$ form a numerical semigroup,
the Weierstrass semigroup.  The space  $\M$ 
parametrising pointed smooth curves with Weierstrass 
semigroup  at the marked point equal to  $\NN$ is a locally  closed subspace of
the moduli space $\mathcal{M}_{g,1}$.
In this paper we compute the dimension of $\M$ for all semigroups of genus at most seven. 

By the famous result of Pinkham \cite{Pi74} the space $\M$ is closely
related to the negative weight part of the versal deformation of the monomial curve
singularity $C_\NN$ with semigroup $\NN$.
This connection has been used in a series of papers
by Nakano--Mori \cite{NM04} and Nakano \cite{Na08, Na16}
to explicitly determine $\M$ for many semigroups of genus at most six, using
the \textsc{Singular} \cite{DGPS} package \texttt{deform.lib} \cite{Ma}.
In all these cases $\M$ is irreducible and rational. For the remaining cases (with
two exceptions)
irreduciblity and stably rationality was shown by Bullock \cite{Bul14}, with different
methods.

We extend the computations of Nakano  \cite{Na08}. 
One quickly runs into the limits of what can be
computed in reasonable time. Therefore we also use other approaches to compute
deformations. One method is to use Hauser's algorithm \cite{Ha83}; 
the method of Contiero-Stöhr \cite{CoSt} to compute $\M$ is closely related.
In this method one first perturbs the equations in all possible ways, and takes care
of flatness only later. 
This means introducing may new variables, most of which
can be eliminated. In a number of cases this approach is succesful.
In one case it is more convenient to use the projection method
developed by De Jong and Van Straten \cite{td2}, as applied to curves in \cite{Ste93}.

We list the semigroups of genus at most 7 in Tables \ref{tab1} and \ref{tab2}. For 
$g\leq 6$ we follow the notation of \cite{Na08}. The corresponding gap sequences 
are already listed by Haure \cite{Hau96}, in the first published paper
containing the term Weierstrass points. Haure also gives the number of moduli
on which curves with given Weierstrass semigroup depend. Our computations shows 
that his results are correct except in one case.  The non-emptiness of
$\M$ for all semigroups with $g\leq7$ was established by Komeda \cite{Kom94}.

Our tables also contain the 
structure of $\M$ in the cases we have been able to determine it.
In many cases, e.g. if the monomial curve $C_\NN$ is a complete intersection,
the space $\M$ is smooth. The next common case is that $\M$ is a weighted
cone over the Segre embedding of $\P^1\times \P^3$; the curve $C_\NN$ 
has then codimension 3 and is given by 6 equations. 
For codimension 4 and 10 equations the base space
is typically  given by 20 equations and the exact structure depends on the curve.
Except for the curve already studied in \cite{Ste93} these equations
are too complicated, with too many monomials, to be useful. 

As to the dimension of $\M$, in general the following bounds are known
\cite{RV77, Co21}:
\[
2g-2+t-\dim\mathrm{T}^{1,+}\leq \dim \M\leq 2g-2+t\;.
\] 
where $t$ is the rank of the highest syzygy module of the ideal of $C_\NN$, and 
$\dim\mathrm{T}^{1,+}$ the number of deformations of positive weight,
both easily computable with \textsc{Singular} \cite{DGPS} or
\emph{Macaulay2} \cite{GrSt}.
The result of our computations is that for all semigroups 
with $g\leq 7$ the dimension is given by the lower bound.

In the first section we recall the relation between the moduli space $\M$
and deformations of the monomial curve with semigroup $\NN$. The next section
describes the computation methods used in this paper. The main part of the paper
discusses the computation of the moduli space or of its dimension for the different 
types of semigroup.

\section{The moduli space $\M$}
\label{Pinksection}
Let $P$ be a smooth point on  a possibly singular integral complete curve $C$  of arithmetic genus $g>1$, 
defined over an algebraically closed field  $\k$
of characteristic zero. An integer $n\in \N$ is a \textit{gap} if 
there does not exist  rational function  on $C$ 
with pole divisor $nP$, or equivalently
$H^0(C, \LO_{C}(n-1)P))=H^{0}(C, \LO_{C}(nP))$. 
There are exactly $g$ gaps
by the  Weierstrass gap theorem, an easy consequence of Riemann-Roch. The nongaps
form a numerical semigroup $\NN$, the \textit{Weierstrass semigroup} of $C$ at $P$; this is the set 
of nonnegative integers  $n\in\N$ such that there is a rational function  on $C$ 
with pole divisor $nP$. For any numerical semigroup the \textit{genus} is defined as the number of gaps.

Given a numerical semigroup $\NN$ of genus $g>1$, let $\M$ be 
the space parameterising pointed smooth curves with $\NN$ as Weierstrass 
semigroup
at the marked point. It is a locally  closed subspace of  the moduli space $\mathcal{M}_{g,1}$
of pointed smooth curves of genus $g$.
Note that $\M$ can be empty.

The connection between the moduli space $\M$ and  deformations of negative weight of 
monomial curves was first observed by Pinkham \cite[Ch. 13]{Pi74}.
Given a numerical semigroup $\NN=\langle n_1,\dots,n_r\rangle$ we
form the semigroup ring $\k[\NN]:=\oplus_{n\in\NN}\k\,t^{n}$
and denote by $C_{\NN}:=\Spec\k[\NN]$ its associated affine monomial curve.
Consider the  versal deformation of $C_{\NN}$
\begin{equation*}
\begin{matrix} 
\mathcal{X}_{t_0}\cong C_{\NN} & \longrightarrow & \mathcal{X} \\[3pt] 
\Big\downarrow & & \Big\downarrow\\[7pt]
\{t_0\}=\Spec\,\k & \longrightarrow & B
\end{matrix}
\end{equation*}
where $B=\Spec A$ is the spectrum of local, complete noetherian $\k$-algebra.
Pinkham \cite{Pi74} showed that the  natural $\mathbb{G}_{m}$-action on $C_{\NN}$
can be extended to the total  and parameter spaces.
This induces a 
grading on the  tangent  space $T^1_{C_{\NN}}$ to
$B$.  The convention here is that a deformation has negative weight $-e$ if it decreases the
weights of the equations of the curve by $e$;  the corresponding deformation variable
has then (positive) weight $e$.
A numerical semigroup $\NN$ is called \textit{negatively
graded} if  $T^{1}_{C_{\NN}}$ has no positive graded part.

Let $B^-$ be the subspace of $B$ with  
negative weights. Then  the restriction $\mathcal{X}^- \to B^-$ is  
versal  for deformations with good $\mathbb{G}_{m}$-action .
Both $\mathcal{X}^- $ and $ B^-$ are defined by polynomials and we 
use the same symbols for the corresponding affine varieties.
The deformation $\mathcal{X}^- \to B^-$ can be fiberwise compactified
to $\smash{\overline{\mathcal{X}}}^- \to B^-$; each fibre is an integral curve
in a weighted projective space with one point $P$ at infinity and this is a point 
with semigroup $\NN$.
All the fibres over a given
$\mathbb{G}_m$ orbit in $\mathcal{T}^-$ are isomorphic, and two fibres
are isomorphic if and only if they lie in the same orbit.
This is proved in \cite{Pi74} for smooth fibres and in general in the Appendix
of \cite{Lo84}.

Each pointed curve from $\M$ occurs as fibre by the following construction.
Consider the section ring $\mathcal{R} = \oplus_{n=0}^\infty H^0(C,\LO(nP))$.
It gives  an embedding of  $C= \Proj \mathcal R$ in a weighted projective space, 
with coordinates $X_0,  \dots, X_r$ where $\deg X_0 = 1$.  
The space $\Spec  \mathcal R$
is the corresponding quasi-cone in affine space. Setting $X_0 = 0$ defines the monomial 
curve $C_\NN$, all other fibres are isomorphic to $C\setminus P$. In particular, if $C$ is 
smooth, this construction defines a smoothing of $C_\NN$.

\begin{thm}[\null{\cite[Thm. 13.9]{Pi74}}]\label{pinkhamthm}
Let $\mathcal{X}^- \to B^-$ be the equivariant 
negative weight miniversal deformation  
of the monomial curve $C_\NN$ for a given semigroup $\NN$ and denote
by $B^-_s$ the open subset of $B^-$ given by the points
with smooth fibers. Then the moduli space $\M$ is isomorphic to the quotient
$\M=B^-_s/\mathbb{G}_{m}$ of  $B^-_s$ 
by the  $\mathbb{G}_{m}$-action.
\end{thm}

The closure of a component of $B^-_s$ is a smoothing component and is itself
contained in a smoothing component in $B$. 
For quasihomogeneous curve singularities there is a simple formula for the
dimension of smoothing components: it is $\mu+t-1$ \cite{Gr82},
with $\mu=2\delta-r+1$  the Milnor number and $t=
\dim_\k \Ext^1_\LO(\k,\LO)$ the type. For monomial curves $\delta=g$ and $r=1$,
and the type can be computed from the semigroup \cite[4.1.2]{Bu80} :
 $t=\lambda(\NN)$, the number of gaps $\l$ of $\NN$ such that $\l+n\in\NN$ 
 whenever $n$ is a nongap. 
Given the equations of $C_\NN$ (anyway needed for deformation computations)
the type is easily found as the rank of the highest syzygy module.
 
Let $\dim\mathrm{T}^{1,+}$ be the dimension of the space of 
infinitesimal deformations of $C_\NN$ of positive weight. Then we have the
following bounds for the dimension of components of $M$ \cite{Co21};
the upper bound is due to Rim--Vitulli \cite{RV77}.

\begin{thm}\label{Stevens}  Let $\N$ be a numerical semigroup $\N$ of genus bigger than $1$. 
If $\M$ is nonempty, then for any irreducible component $E$ of $\M$  
\[
2g-2+t-\dim\mathrm{T}^{1,+}\leq \dim E\leq 2g-2+t\;.
\]
\end{thm}

\section{Computing deformation spaces in negative weight}
By Pinkham's theorem (Theorem \ref{pinkhamthm}), to explicitly describe the moduli space $\M$ one can compute the negative weight part of the
versal deformation  of the monomial curve $C_\NN$.
For many semigroups of low genus this was done by Nakano-Mori \cite{NM04} and
Nakano \cite{Na08,Na16}, using the computer algebra system \textsc{Singular} \cite{DGPS}.
The main obstacle in the remaining cases is that the computations take too long,
and result in long formulas without apparent structure.
In this section we describe several methods
to determine versal deformations,
with comments on computational matters. 

\subsection{The standard approach}
We recall the main steps, see also \cite[Ch. 3]{Ste03}.
Let $X$ be a variety with $\mathbb{G}_{m}$-action with isolated singularity
at the origin in $\A^n$. Let $S=\k[X_1,\dots,X_n]$ be the polynomial ring in 
$n$ variables. Let $f=(f_1,\dots,f_r)$ generate the ideal $I(X)$ of $X$.
The first few terms of the resolution of $\k[X]=S/I(X)$ are
$$
0\longleftarrow \k[X]\longleftarrow S 
\mapleft f S^k \mapleft r S^l\;
$$
where the columns of the matrix $r$ generate the module of relations.
Let $X_B\to B$
be a deformation of $X$ over $ B=\Spec A$. The flatness of the map 
 $X_B\to B$
translates into the existence of a lifting of the resolution to
$$
0\longleftarrow \k[X_B]\longleftarrow S\otimes A 
\mapleft F (S\otimes A)^k \mapleft R (S\otimes A)^l\;.
$$
To find the versal deformation we must find a lift $FR=0$ in the most general way.
The first step is to compute infinitesimal deformations. We write 
$F=f+\ep f'$ and $R=r+\ep r'$. As $\ep^2=0$, the condition
$FR=0$ gives
\[
FR=(f+\ep f')(r+\ep r')=fr+\ep(fr'+f'r)=0\;.
\]
Because $fr=0$, we obtain the equation $fr'+f'r=0$ in $S$. We first solve the equation
$f'r=0$ or rather its transpose $r^t(f')^t=0$ in $\k[X]$. This means finding syzygies
between the columns of the matrix $r^t$; then we find $r'$ by lifting $f'r$ with $f$. 
After this we lift order for order. Obstructions to do this may come up,
leading to equations in the
deformation parameters. 
\comment 
We consider now $\ep$ as a parameter
marking the order and suppose that $f'$ depends on deformation parameters.
Suppose we have found $F_{n-1}$ and $R_{n-1}$ with 
$F_{n-1}R_{n-1}\equiv 0 \pmod {\ep^n}$. We want a solution modulo
$\ep^{n+1}$, so we put $F_n=F_{n-1}+\ep^nf^{(n)}$,
$R_n=R_{n-1}+\ep^nr^{(n)}$. Then
$$ 
\displaylines{\qquad
F_nR_n\equiv F_{n-1}R_{n-1}+ \ep^n\big(f^{(n)}R_{n-1}+F_{n-1}r^{(n)}\big)
\hfill\cr\hfill{}
\equiv F_{n-1}R_{n-1}+\ep^n\big(f^{(n)}r+fr^{(n)}\big) \equiv 0 \pmod
{\ep^{n+1}}\;.\qquad\cr} 
$$
Modulo $f$ we have to find $f^{(n)}$ from
$$
\ep^{-n} F_{n-1}R_{n-1}+f^{(n)}r\equiv 0 \pmod {\ep}\;.
$$
This is possible if the normal form of the column vector 
$\ep^{-n} F_{n-1}R_{n-1}$ with respect to the module generated by
the columns of the matrix $r$ is zero.
Here obstructions may come up, which lead to equations in the
deformation parameters: the expressions depending on the
deformation parameters in the normal form have to vanish. 
The step is concluded by finding $r^{(n)}$.
\endcomment

All these computations can be  done with a computer algebra system. Indeed,
they are implemented
implemented \cite{Il,Ma} in   \emph{Macaulay2} \cite{GrSt}
and \textsc{Singular} \cite{DGPS}.
The specific outcome of a computation, which depends  on Groebner basis calculations, 
is governed by the chosen monomial ordering and
also by the choice of the generators $(f_1,\dots,f_k)$ of the ideal
$I(X)$.  
The algorithm tries to find the row vector $F$, equations of the
base space come from obstructions to do that. Typically
a computer computation will not choose the easiest form of the base 
equations.

When restricting to deformations of negative weight all resulting
equations are polynomial and the computation is finite; it might be undoable
in practice, even with a powerful computer.

\subsection{Hauser's algorithm}\label{hauser}
An alternative method was 
developed in the complex analytic setting by Hauser \cite{Ha83,Ha85}.
One can see the method of Contiero-Stöhr \cite{CoSt} to compute 
a compactification of the moduli space $\M$ as a variant.
It has been used  in \cite{Co21} to compute the base space
for several families of Gorenstein monomials curves.
We start again from the generators $f$ of the ideal $I(X)$,
but now we perturb $f$ in the most general way, modulo trivial
perturbations, that is we take a semi-universal unfolding 
of the associated map $f\colon \A^n \to \A^k$. Except when
$X$ is zero dimensional, the base space of this unfolding will
be infinite dimensional; this problem is handled carefully by Hauser  \cite{Ha85}.
In our situation $f$ is weighted homogeneous and we restrict
ourselves to an unfolding with terms of lower degree. Therefore we
are back in a finite dimensional situation, and we can work
over any field $\k$. So we have an unfolding
$F\colon \A^n\times \A^s \to \A^k$ of $f$. Now we determine the 
locus $B$ containing $0\in \A^s$ over which $F$ is flat. The
restriction of $F$ to $B$ is then the versal deformation of $X$
of negative weight.

In our situation we have 
a monomial curve $X$ of multiplicity $m$ and embedding dimension
$e+1$. We take coordinates $x,y_1,\dots,y_e$. An Apéry basis of the 
semigroup leads to an additive realisation of $\k[X]$ as 
$\sum_{i=0}^m y^{(i)}\k[x]$, where $y^{(0)}=1$, $y^{(1)}\dots y^{(m-1)}$ are
expressions in the variables $y_1,\dots,y_e$. The equations of $X$
are then (in multi-index notation) of the form $y^\alpha=\varphi$
with $\varphi \in \sum_{i=0}^m y^{(i)}\k[x]$. The unfolding is also done
only with terms from $\sum_{i=0}^m y^{(i)}\k[x]$.
We start from generators $f$ of the said form, compute the relation matrix $r$
and write the unfolding $F$. We have to lift $fr=0$ to $FR=0$.
To this end we compute $Fr$ and reduce this column vector to normal form with
respect to the list $F$. It is important that we do not compute
a Groebner basis of the ideal generated by  $F$, as this will take too long.
But reducing with respect to $F$ will result in a vector with entries of bounded degree
lying 
in  $\sum_{i=0}^m y^{(i)}R[x]$ with coefficients from $R=\k[t_1,\dots,t_s]$,
where the $t_j$ are coordinates on the base $\A^s$. The vanishing of these
coefficients define the locus where $FR=0$, so where $F$ is flat.

This procedure  leads to a rather large number
of relatively simple equations in a large number of variables, most of which
occur linearly and can be eliminated. It is this process of elimination
which can lead to few equations in a limited number of variables, but with 
many monomials, see the proof of Proposition \ref{propdim} for an example.

Also here most computations are easily done with a computer
algebra system. The first step, to find the unfolding, can be 
automatised, but for the not too complicated
cases relevant for this paper  it seems preferable to do it by hand,
choosing names for the deformation variables reflecting their weights.

\subsection{The projection method}
Computing  deformations using projections onto a hypersurface
is a method developed in a series of papers by Theo de Jong and Duco van Straten,
see  \cite{td1,td2}.
The application to curves is in \cite{Ste93}, see also \cite[Ch. 11]{Ste03}.
Let again $X$ be a monomial curve and $X\to Y$ a projection onto a plane curve,
which is a finite generically injective map. Let 
$\Sigma$ be the subspace of $Y$  by the conductor
ideal  $I={\Hom}_Y({\mathcal O}_{X}, {\mathcal O}_Y)$  in ${\mathcal O}_Y$.
This makes it possible  to reconstruct
$X$, as ${\mathcal O}_{X}={\Hom}_Y(I, {\mathcal O}_Y)$. 
Because we use this method only once, we refer to \cite[Chapter 11]{Ste03} for a 
description how to use deformations the plane curve $Y$ together
with  $\Sigma$ to get deformations of the original curve $X$.
\comment
The fact that ${\mathcal O}_{X}$ has a ring
structure, is equivalent to the 
{\sl Ring Condition\/}
$$
{\rm (R.C.):}\qquad 
{\Hom}_Y({\mathfrak c},{\mathfrak c})\buildrel \approx \over
{\lhook\joinrel\longrightarrow}{\Hom}_Y({\mathfrak c}, {\mathcal O}_Y)\;.
$$  
This condition makes sense
over any basis $B$, so we can set up a deformation theory 
with 
${\rm Def}(\Sigma\hookrightarrow Y, {\rm
R.C.})$  the functor of deformations for which the ideal of
$\Sigma_B$ in $Y_B$ satisfies the condition $({\rm R.C.})$.
In our situation, but also under more general assumptions \cite{td2},
this functor is equivalent to the functor ${\rm Def}(X\to Y)$ of deformations
of the map $X\to Y$, and the natural transformation
${\rm Def}(X\to Y)\longrightarrow{\rm Def}(X)$ is smooth.
\endcomment

The space $\Sigma$ is a fat point, so in particular Cohen-Macaulay of
codimension 2. Therefore the ideal $I$ defining $\Sigma$ in $\A^2$ is
generated by the maximal
minors  ${\Delta_1}\dots{\Delta_k}$ of an $k\times(k-1)$ matrix
$M$. We write these generators as row vector $\Delta$. The curve $Y$
is defined by a function of the form $f=\Delta \alpha$ with $\alpha$
a column vector, or equivalently by the determinant of the matrix $(M,\alpha)$. 
We write an element $n\colon
\Delta_i\mapsto n_i$ of the normal module  
$N:= {\rm Hom}_{\Sigma}(I/I^2,{\mathcal O}_{\Sigma})$  as row vector $n$. 
\comment
Let $N:= {\rm Hom}_{\Sigma}(I/I^2,{\mathcal O}_{\Sigma})$ be
the normal module of $\Sigma$. We write an element $n\colon
\Delta_i\mapsto n_i$ of $N$ as row vector $n$. 
There is an {\sl evaluation
map\/} $\ev_f\colon N\to {\mathcal O}_{\Sigma}$, given by $n\mapsto n(f)$. In
terms of $\alpha$ we can write $\ev_f(n)=n\alpha$. This map descends to a map
$\ev_f\colon T^1_\Sigma\to {\mathcal O}_{\Sigma}$. Let now 
$\Sigma_B\to Y_B$ be any deformation over a base
$(B,0)$, with $Y_B$ defined by a function $F$.
By  \cite[(1.12)]{td2} the ideal $I_B$ of $\Sigma_B$ satisfies the condition \/{\rm
(R.C.)} if and only if the map $ev_F$ is the zero map. 
This means that for every normal vector $n_B$ there
exists a vector $\gamma_B$ on the ambient space,
satisfying  
\endcomment
A deformation $\Sigma_B\to Y_B$ 
 comes from
a deformation of the curve $X$ if for every  normal vector $n_B$ there
exists a vector $\gamma_B$ on the ambient space,
satisfying  
\begin{equation}\label{basic}
n_B\alpha_B+\Delta_B\gamma_B=0\;.
\end{equation}
This is the basic deformation equation, which can be solved step by step.
\comment
Infinitesimal deformations are given by perturbing $f$ to $f+\ep g$
with $g\in I$ satisfying $ ev_g=0$, and perturbing $\Delta$ such the
equation \eqref{basic} is satisfied. Again the lift to higher order
is done step for step, with possible obstructions. 
\endcomment

When restricting to deformations of negative weight the result
of the computation is again given by quasihomogeneous matrices
with polynomial entries. Once setup correctly the computation
is easily done with a computer algebra system.

An important concept here is that of $I^2$-equivalence \cite[Def.~1.14]{td1}: 
two functions $f$ and $g$
are $I^2$-equivalent, if and only if $f-g\in I^2$. 
Suppose $f=\Delta\alpha$ and $g=\Delta\beta$ are $I^2$-equivalent.
Then $\alpha-\beta=A \Delta^t$ for some matrix $A$.
Suppose 
$n_B\alpha_B+\Delta_B\gamma_B$ is a lift of
$n\alpha+\Delta\gamma$ over a base space $B$. Choose any lift $A_B$ of $A$.
Then
\begin{equation}\label{ikweq}
n_B(\alpha_B-A_B\Delta_B^t)+\Delta_B(\gamma_B+A_B^tn_B^t)
\end{equation}
is a  lift of $n\beta+\Delta(\gamma+A^tn^t)$.
\comment
Suppose $f_B=\Delta_B\alpha_B$ and $g_B=\Delta_B\beta_B$ are $I^2$-equivalent.
Then $\alpha_B-\beta_B=A_B\Delta_B^t$ for some matrix $A_B$.
Suppose 
$n_C\alpha_C+\Delta_C\gamma_C$ is a lift of
$n_B\alpha_B+\Delta_B\gamma_B$ over a larger base space $C$. Choose any lift $A_C$ of $A_B$.
Then
$n_C(\alpha_C-A_C\Delta_C^t)+\Delta_C(\gamma_C+A_C^tn_C^t)$ is
a  lift of $n_B\beta_B+\Delta_B(\gamma_B+A_B^tn_B^t)$.
\endcomment
In particular, for curves with projections defined by $I^2$-equivalent functions,
the base spaces of the versal deformation are the same up to a smooth factor.

\section{Semigroups of genus $g\leq7$}
In Tables \ref{tab1} and \ref{tab2} we  list the semigroups of genus at most 7. For 
$g\leq 6$ we follow the notation of \cite{Na08}. The tables also contain 
also the dimension $d$ of $\M$ and the type $t$ of the semigroup. Furthermore
they give under the heading base the 
structure of $\M$ in the cases we have been able to determine it; the entries 
indicating 
the different possibilities are discussed below.
Inspection of the tables shows that the main parameters
governing the structure of $\M$ are the number of generators of $\NN$
and the type $t$.
The first step in our computations is always to find equations 
for the monomial curve $C_\NN$, followed by the free resolution. This gives the type 
$t$. The next step is to find the graded parts of  the vector space $T^1$
of infinitisemal deformations.

\begin{prop}\label{dimprop}
For all semigroups with $g\leq7$ the dimension of $\M$ is given by the lower bound
$2g-2+t-\dim\mathrm{T}^{1,+}$ of Theorem {\rm \ref{Stevens}}.
\end{prop}
\noindent This result follows from the computations discussed in the rest of this section.

\begingroup
\small
\begin{table}[htb]
\begin{center}
\begin{tabular}{llccc||llccc}
\cite{Na08}&  semigroup& $d$& $t$ & base&\cite{Na08}& semigroup& $d$& $t$ & base\\ \hline \hline
$N(1)_{1} $ & 2,3 &  1  & 1 & sm &
$N(6)_{1} $ & 2,13 &11 & 1  & sm\\  

$N(2)_{1} $ & 2,5 &  3 & 1 & sm &   
$N(6)_{2} $ & 3,10,11 & 12 & 2 &sm \\  
$N(2)_{2} $ & 3,4,5 &  4 & 2 & sm & 
$N(6)_{3} $ & 3,8,13 & 11 & 2 & sm \\  

$N(3)_{1} $ & 2,7 & 5 & 1 & sm &  
$N(6)_{4} $ & 3,7 & 10 & 1 & sm \\  
$N(3)_{2} $ & 3,5,7 &  6& 2 & sm&   
$N(6)_{5} $ & 4,9,10,11& 13 & 3& $B_2$\\  
$N(3)_{3} $ & 3,4 & 5 & 1 &sm &     
$N(6)_{6} $ & 4,7,10,13 & 12 & 3 & $B_1$ \\ 
$N(3)_{4} $ & 4,5,6,7& 7 & 3 & $B_1$& 
$N(6)_{7} $ & 4,7,9 & 11 &2 & sm \\  

$N(4)_{1} $ & 2,9 &  7 & 1 & sm&  
$N(6)_{8} $ & 4,6,11,13&  11 & 3 &$B_1^*$ \\  
$N(4)_{2} $ & 3,7,8 &8 & 2 & sm&  
$N(6)_{9} $ & 4,6,9 &  10 & 1 &sm \\   
$N(4)_{3} $ & 3,5&  7 & 1 &sm &  
$N(6)_{10} $ & 4,5 & 10 & 1 & sm \\  
$N(4)_{4} $ & 4,6,7,9 & 9 &3 & $B_1$ &  
$N(6)_{11} $ & 5,8,9,11,12 &  14 & 4 & ? \\  
$N(4)_{5} $ & 4,5,7 & 8 & 2 &sm & 
$N(6)_{12} $ & 5,7,9,11,13 &   13  & 4 & ! \\  
$N(4)_{6} $ & 4,5,6 &  7 & 1 &sm &  
$N(6)_{13} $ & 5,7,8,11 & 12 & 3 & $B_1$  \\  
$N(4)_{7} $ & 5,6,7,8,9 & 10 & 4 & !& 
$N(6)_{14} $ & 5,7,8,9 & 11 & 1 & sm \\ 
\cline{1-5}
$N(5)_{1} $ & 2,11 & 9 & 1 & sm& 
$N(6)_{15} $ & 5,6,9,13 & 12 &3 & $B_1$  \\  
$N(5)_{2} $ & 3,8,10 &  10 &  2 & sm &  
$N(6)_{16} $ & 5,6,8 &  11 & 2 & sm \\ 
$N(5)_{3} $ & $3,7,11$ &  9 &  1 & sm &  
$N(6)_{17} $ & 5,6,7 & 10 & 2 & sm \\ 
$N(5)_{4} $ & 4,7,9,10 & 11  & 3 & $B_1$  & 
$N(6)_{18} $ & 6,8,9,10,11,13& 15 & 5 &? \\ 
$N(5)_{5} $ & 4,6,9,11 & 10  & 3 & $B_1$  & 
$N(6)_{19} $ & 6,7,9,10,11 & 14 & 4  & ? \\  
$N(5)_{6} $ & 4,6,7 & 9&  1 & sm &  
$N(6)_{20} $ & 6,7,8,10,11& 13 & 3 & $G'$ \\   
$N(5)_{7} $ & 4,5,11 & 9 & 2 & sm &  
$N(6)_{21} $ & 6,7,8,9,11& 12 & 2 & $G$ \\  
$N(5)_{8} $ & 5,7,8,9,11 & 12 & 4& ?&  
$N(6)_{22} $ & 6,7,8,9,10&  11 & 1 & $G$\\  
$N(5)_{9} $ & 5,6,8,9 &11 & 3 & $B_1$  & 
$N(6)_{23}  $ & $7, \dots,13$ &16 & 6  & ? \\  
$N(5)_{10} $ & 5,6,7,9 & 10 & 2& sm \\  
$N(5)_{11} $ & 5,6,7,8 &  9 & 1 & sm \\  
$N(5)_{12} $ & $6, \dots,\!11$ & 13 &5 & ? \\  
\hline

\end{tabular}
\caption{semigroups of genus $\leq 6$}
\label{tab1}
\end{center}
\end{table}
\endgroup
\subsection{Smooth base space}
The base space of the versal deformation of $C_\NN$ is smooth (indicated with
\lq sm\rq\  in the tables) if
the obstruction space $T^2$ vanishes. This happens if the curve is
a complete intersection, or Gorenstein of codimension three, or 
Cohen-Macaulay of codimension two. In the latter case the equations are the vanishing
the minors of a $2\times 3$ matrix, and $t=2$. Also $T^2=0$ for codimension 3 
curves with $t=2$; two of these curves, $N(7)_{31}$ and
$N(7)_{32}$ are almost complete intersections.

\begingroup
\small
\begin{table}[htb]
\begin{center}
\begin{tabular}{l>{\footnotesize}lccc||l>{\footnotesize}lccc}
name& \small semigroup& $d$& $t$ & base&name& \small semigroup& $d$& $t$ & base\\ \hline \hline
$N(7)_{1} $ & 2,15 &13 & 1  & sm & 
$N(7)_{21} $ & 5,6,9 &12 & 2 & sm\\  
$N(7)_{2} $ & 3,11,13 &14 & 2  & sm&  
$N(7)_{22} $ & 6,9,10,11,13,14&17 & 5 & ?\\  
$N(7)_{3} $ & 3,10,14 &13 & 2  & sm&  
$N(7)_{23} $ & 6,8,10,11,13,15 &16 & 5 & ?\\  
$N(7)_{4} $ & 3,8 &12 & 1  & sm&  
$N(7)_{24} $ & 6,8,9,11,13 &15 & 4 & ?\\  

$N(7)_{5} $ & 4,10,11,13 &15 & 3 & $B_2$&  
$N(7)_{25} $ & 6,8,9,10,13 &14 & 2  & $G'$\\  
$N(7)_{6} $ & 4,9,11,14 &14 & 3 & $B_1$&  
$N(7)_{26} $ & 6,8,9,10,11 &13 & 1  & $G$\\  
$N(7)_{7} $ & 4,9,10,15 &13 & 3 & $B_1$&  
$N(7)_{27} $ & 6,7,10,11,15&15 & 4 & ?\\  

$N(7)_{8} $ & 4,7,13 &13 & 2  & sm&  
$N(7)_{28} $ & 6,7,9,11 &14 & 3 & $B_1$\\  
$N(7)_{9} $ & 4,7,10 &12 & 1  & sm&  
$N(7)_{29} $ & 6,7,9,10 &13 & 3 & $B_1$\\  
$N(7)_{10} $ & 4,9,11,14 &12 & 3 & $B_1^*$&  
$N(7)_{30} $ & 6,7,8,11 &13 & 3  & $B_1$\\  
$N(7)_{11} $ & 4,6,11 &10 & 1 & sm&  
$N(7)_{31} $ & 6,7,8,10 &12 & 2  & sm \\  

$N(7)_{12} $ & 5,9,11,12,13 &16 & 4 & ?&  
$N(7)_{32} $ & 6,7,8,9 &11 & 2 & sm\\  
$N(7)_{13} $ & 5,8,11,12,14 &15 & 4 & ?&  
$N(7)_{33} $ & $7,9,\dots,13,15$&18 & 6 & ?\\  
$N(7)_{14} $ & 5,8,9,12 &14 & 3 & $B_1$&  
$N(7)_{34} $ & 7,8,10,11,12,13 &17 & 5 & ?\\  
$N(7)_{15} $ & 5,8,9,11 &13 & 2  & sm&  
$N(7)_{35} $ & 7,8,9,11,12,13 &16 & 4 & ?\\  

$N(7)_{16} $ & 5,7,11,13 &14 & 3 & $B_1$&  
$N(7)_{36} $ & 7,8,9,10,12,13 &15 & 3  & ?\\  
$N(7)_{17} $ & 5,7,9,13 &13 & 2 & sm&  
$N(7)_{37} $ & 7,8,9,10,11,13 &14 & 2  & ? \\  
$N(7)_{18} $ & 5,7,9,11 &12 & 1 & sm&  

$N(7)_{38} $ & 7,8,9,10,11,12 &13 & 1 & ?\\  
$N(7)_{19} $ & 5,7,8 &12 & 2  & sm&  

$N(7)_{39} $ & $8,\dots,15$ &19 & 7 & ?\\  

$N(7)_{20} $ & 5,6,13,14 &13 & 3  & $B_1$  
\\ \hline

\end{tabular}
\caption{semigroups of genus $7$}
\label{tab2}
\end{center}
\end{table}
\endgroup

\subsection{Cone over a Segre embedding}
For codimension 3 curves $C_\NN$ with  $t=3$ and $g\leq7$  it is possible
to explicitly determine the structure of the base space.
Most cases with $g\leq 6$ were computed by Nakano \cite{Na08}. 

The result for $N(6)_{6}=\langle 4,7,10,13\rangle $ was not given in \cite{Na08}.
We computed the deformation using Bernd Martin's \textsc{Singular} \cite{DGPS}
package
\texttt{deform.lib} \cite{Ma}.
The equations for the curve are determinantal:
\[
\begin{bmatrix}
x & \y(7) & \y(10) & \y(13) \\
 \y(7) & \y(10) & \y(13) & x^4
\end{bmatrix}
\]
The speed of the computation in \textsc{Singular} depends very much on the
chosen ordering. A good choose is using the variables
$( \y(13) , \y(10) , \y(7), x)$ in this order with graded reverse lexicographic
order, but with weights of the variables all equal to 1, not using the weights
$13,10,7,4$.
\textsc{Singular} returns an ideal $Js$ in 16 variables $A,\dots,P$ of weight
$2$, $6$, $10$, $1$, $5$, $9$, $13$, $8$, $12$, $16$, $7$, $10$, $4$, $1$, $4$, $7$.
It is generated by the minors of
\begin{equation}
\label{matrixeq}
\begin{bmatrix}
N & -K+P\\
O & C-L-BM+AM^2\\
P & 2G-2FM+2EM^2-2DM^3-M^3N\\
C-BM+AM^2 &-J+IM-HM^2-M^4+M^3O
\end{bmatrix}
\end{equation}
We are allowed to simplify the equations of the base space by a coordinate
transformation.
An obvious transformation gives the matrix of
the cone over the Segre embedding of $\P^1\times\P^3$.

Observe that the coordinate ring of the Segre cone has 
a resolution of the form
$$
0\longleftarrow  S/I \longleftarrow S 
\mapleft f S^6 \mapleft r S^8 \mapleft s S^3 \longleftarrow 0
$$
where $f$ is the row vector of minors of the matrix
\[
\begin{bmatrix}
X_1 & X_2 & X_3 & X_4 \\
Y_1 & Y_2 & Y_3 & Y_4
\end{bmatrix}
\]
and the $8\times3$ matrix $s$ is the transpose of a matrix of the form
\[
\begin{bmatrix}
X_1 & X_2 & X_3 & X_4 & 0 &0 & 0 & 0\\
Y_1 & Y_2 & Y_3 & Y_4 & X_1 & X_2 & X_3 & X_4\\ 
0 &0 & 0 & 0& Y_1 & Y_2 & Y_3 &Y_4
\end{bmatrix}
\]
Computing the resolution of the ideal $Js$ with \textsc{Singular}
gives indeed a $8\times3$ matrix with some zeroes,
but of course not exactly in the form above. This form
can be achieved by column and row operations; in this way the matrix 
\eqref{matrixeq} was found.

The Segre cone  occurs for many curves as base space. A necessary condition is that 
$\dim T^2=6$. For some curves $\dim T^2 = 12$. This happens for
the semigroups $\langle n_1,n_2,n_3,n_4\rangle$ in the tables with $n_2>2 n_1$.
Then the base space has a more complicated structure. 

\begin{prop}
For monomial curve singularity of genus at most $7$ of codimension $3$, with type
$t=3$, such that the first blow up has
lower embedding dimension, the base space of negative weight is up to a smooth factor the
cone over the Segre embedding of\/ $\P^1\times\P^3$,
except in the cases $N(6)_{8} $ and $N(7)_{10} $, where it has two components,
being the intersection of the Segre cone with  coordinate hyperplanes {\rm(}$B_1^*$ in
Tables {\rm \ref{tab1}} and {\rm\ref{tab2}}{\rm)}.
\end{prop}

\begin{proof}
By the  assumption in the statement $\dim T^2=6$.
Most cases with $g\leq 6$ were computed by Nakano \cite{Na08}. 
It can be checked that systems of generators 
given in \cite{Na08} are minors of $2\times4$ matrices.

To identify the base space as Segre cone it in fact suffices to 
show that the quadratic part of the equations defines a Segre cone.
If it does, the Segre cone has to be the tangent cone of the base 
space, for otherwise the dimension of the tangent cone should be 
less than the dimension of the Segre cone, but this dimension
is in all cases
equal to the lower bound of theorem \ref{Stevens}. Because the Segre cone is rigid
and the base space itself is a deformation of its tangent cone, they are
isomorphic.

For the semigroups with $4$ generators and $t=3$ of genus 
$g=7$ (see Table \ref{tab2}) the versal deformation 
up to degree 2 is easily computed with \textsc{Singular}, and 
the cases where the base space is $B_1$ identified.

For $N(6)_8=\langle 4,6,11,13\rangle$ the versal deformation
in all degrees can be computed. The generators of the ideal of the curve are the
$2\times2$ minors of the symmetric matrix
\[
\begin{bmatrix}
x & \y(6) & \y(11)  \\
\y(6) & x^2 & \y(13) \\
\y(11) & \y(13) & x^3\y(6)
\end{bmatrix}
\]
With these generators and the graded reverse lexicographic order
with variables $(\y(13),\y(11),\y(6),x)$ of weights $(13,11,6,4)$
\textsc{Singular} succeeds in computing rather quickly the versal deformation
in all degrees. Replacing the generator $\y(11)\y(13)-x^3\y(6)^2$ by
$\y(11)\y(13)-x^6$ results in a computation which does not finish in 
reasonable time.
After a coordinate transformation the base space is given by the minors of 
\[
\begin{bmatrix}
T_{-1} & T_1 & T_6 & T_8 \\
T_{9} & T_{11} & T_{16} & T_{18}
\end{bmatrix}
\]
where the indices indicate the weight of the deformation variables.
The base space is again a Segre cone, but the base space in negative
weight lies in the hyperplane $T_{-1}=0$ and consists of two components,
one smooth given by $T_1 = T_6 = T_8=0$ and the other by $T_9=0$ and the 
vanishing of the three minors of the matrix not involving $T_9$.
Note that the last generator of the ideal of the second component 
given in \cite[p.159]{Na08} can be expressed in the previous ones.

For  $N(7)_{10}=\langle 4,6,13,15\rangle$ the quadratic part of the equations for the
bases space are the minors of 
\[
\begin{bmatrix}
T_{-3} & T_{-1} & T_6 & T_8 \\
T_{11} & T_{13} & T_{20} & T_{22}
\end{bmatrix}
\]
and in negative weight there are two components of different dimension,
given by $T_{11} =T_{13} =T_6T_{22}-T_8T_{20}=0$ and $T_6=T_8=0$.
Over the largest component the equations of the total space can be written in
rolling factors format (see e.g. \cite[p. 95]{Ste03}):
three equations are the minors of the matrix
\[
\begin{bmatrix}
x & \y(6) + T_2 x & \y(13)  \\
\y(6) & x^2 +T_4 x& \y(15) 
\end{bmatrix}
\]
while the fifth and sixth equations are obtained by replacing  in each monomial
a factor occurring in the top row of the matrix by one of the bottom row.
From the equations in the matrix one finds
that $x(x\y(13) +T_4 \y(13))=\y(6)(\y(15)+T_2\y(13))$,  so $\y(15)+T_2\y(13)$ rolls
to $x\y(13) +T_4 \y(13)$. This gives:
\begin{alignat*}{2}
x^2\y(6)^3&-\y(13)^2&&+P_{22}x+P_{20}\y(6)+T_{13}\y(13)+T_{11}\y(15)\\
x\y(6)^4&-\y(13)\y(15)&&+P_{22}\y(6) +\dots\\
\y(6)^5&-\y(15)^2&&+P_{22}(x^2 +T_4 x-T_2\y(6))+\dots
\end{alignat*}
Here $P_{22}$ and $P_{20}$ are polynomials containing deformation
variables of degree $4,6,\dots, 22$. 
It follows in particular that the origin is a singular point of all fibres, in general
an ordinary double point. Only the smallest component is a smoothing component.
\end{proof}

\begin{remark}
For $N(7)_{10}$ the gap sequence is $1,2,3,5,7,9,11$ and Haure
\cite{Hau96} gives a plane model of degree 13 with 11 moduli, whereas 
$\dim \M = 12$. This is the only case where Haure's result differs from our result.
\end{remark}

\subsection{Curves with first blow-up of multiplicity four}
For the cases $N(6)_{5} =\langle 4, 9, 10, 11 \rangle$
and  $N(6)_{5} =\langle 4, 10, 11,13 \rangle$ the first blow-up
is $N(3)_{4} =\langle 4, 5, 6, 7 \rangle$ and $N(4)_{4} =\langle 4, 6, 7,9 \rangle$
respectively.
For the first curve we  compute the base space with Hauser's algorithm; we do it in fact
for all semigroups $\langle 4,1+4\tau ,2+4\tau ,3+4\tau  \rangle$. A similar, but more complicated computation is in \cite{Co21}.

The equations of the curve are given by the minors of the matrix
\[
\begin{bmatrix}
x^\tau &    \y(1)& \y(2)& \y(3)\\
\y(1)& \y(2)& \y(3)& x^{\tau +1}
\end{bmatrix}
\]
We write the unfolding with variables which are polynomials in $x$,
where $\f(i)(j)$ with $i\neq j$ has degree $4\tau +j$. We use coordinate transformations
to remove as many terms as possible. The result is
\begin{align*}
\y(1)^2-\y(2)x^\tau         &+\f(2)(1)\y(1)+\f(2)(2)+\f(2)(-1)\y(3)+\f(2)(0)\y(2)\\
\y(1)\y(2)-\y(3)x^\tau      &+\f(3)(1)\y(2)+\f(3)(2)\y(1)+\f(3)(3)+\f(3)(0)\y(3)\\
\y(1)\y(3)-x^{2\tau +1}                 &+\f(4)(2)\y(2)+\f(4)(3)\y(1)+\f(4)(4)\\
\y(2)^2-x^{2\tau +1}      & +\g(4)(1)\y(3)+\g(4)(2)\y(2)+\g(4)(3)\y(1)+\g(4)(4)\\
\y(2)\y(3)-\y(1)x^{\tau +1} &+\f(5)(5)                          +\f(5)(4)\y(1)\\
\y(3)^2-\y(2)x^{\tau +1}     & +\f(6)(5)\y(1)+\f(6)(6)+\f(6)(3)\y(3)+\f(6)(4)\y(2)
\end{align*}
We have  four transformations left, which we cannot show in the above notation.
They act on the unfolding as 
$\f(3)(1)-x^\tau a_{3,1}$,  $\f(2)(1)-x^\tau a_{2,1}$, $\f(4)(2)+x^\tau a_{3,2}$
and $\f(2)(0)+\tau a_0x^{\tau -1}$; we use them to remove the 
lowest weight variables from  $\f(2)(0)$, $\f(2)(1)$,  $\f(3)(1)$ and  $\f(4)(2)$.

We proceed as explained in section \ref{hauser}: 
we compute the relation matrix for the unperturbed
generators of the ideal, multiply with the perturbed generators and reduce the result 
with them. The result does not contain quadratic monomials in the $\y(i)$ and for flatness
it has to vanish identically, giving conditions on the coefficients.
We write these as equations for the polynomials $\f(i)(j)$, $\g(i)(j)$.
The polynomials $\f(i)(i)$  and $\g(4)(4)$ can be eliminated. We obtain 	fifteen equations.

The first one is 
$(x^\tau -\f(3)(0))(\f(2)(1)-\f(3)(1))+
(x^\tau -\f(2)(0))\g(4)(1)+\f(2)(-1)\f(3)(2)=0$.
We will use this equation to eliminate $\g(4)(1)$. To this end we rewrite
it, and do the same with five other equations containing $x^\tau -\f(2)(0)$. We obtain
\begin{multline*}
(x^\tau -\f(2)(0))(\g(4)(1)+\f(2)(1)-\f(3)(1))=
{ -(\f(2)(0)-\f(3)(0))(\f(2)(1)-\f(3)(1))-\f(2)(-1)\f(3)(2)}\\
\shoveleft{(x^\tau -\f(2)(0))(\f(3)(2)-\g(4)(2))=
    -(\f(2)(1)-\f(3)(1))\f(3)(1)-\f(2)(-1)(\f(4)(3)-\f(6)(3))}\\
\shoveleft{(x^\tau -\f(2)(0))(\g(4)(3)-\f(4)(3)+\f(6)(3)) =
  (\f(2)(0)-\f(3)(0))(\f(4)(3)-\f(6)(3))-\f(3)(1)\f(3)(2)}\\
\shoveleft{(x^\tau -\f(2)(0))(\f(4)(3)-x\f(2)(-1)) =  (x\f(2)(0)-\f(6)(4))\f(2)(-1)-(\f(2)(1)-\f(3)(1))\f(4)(2)}\\
\shoveleft{(x^\tau -\f(2)(0))(\f(6)(4)-\f(5)(4)-x\f(2)(0)+x\f(3)(0)) }\\
   \shoveright{ =\f(3)(2)\f(4)(2)+(\f(2)(0)-\f(3)(0))(x\f(2)(0)-\f(6)(4))}\\
\shoveleft{(x^\tau -\f(2)(0))(\f(6)(5)+x\f(3)(1))=
 - (x\f(2)(0)-\f(6)(4))\f(3)(1)-\f(4)(2)(\f(4)(3)-\f(6)(3))}
  \end{multline*}
It can be checked that the remaining equations are consequences of these ones.
All the above equations are of the form
\[
L\cdot (x^\tau - \f(2)(0))=R
\]
with $L$ and $R$ polynomials in $x$ satisfying
$\deg_x(R)\leq \deg_x(L)+t$. Division with remainder gives
$R=Q (x^\tau - \f(2)(0))+\overline R$, and therefore we can solve $L=Q$ and
find the coefficients of $\overline R$ as equations for the base space. 
In other words, the condition leading to the equations of the base
space is that the right hand side of the above equations is divisible by
$x^\tau -\f(2)(0)$. A similar structure first appeared for the base spaces of rational 
surface singularities of multiplicity four \cite{td3}.

The right hand side of the equations are the minors of the matrix
\[
\begin{bmatrix}
-\f(2)(-1)          & (\f(2)(0)-\f(3)(0))  &\f(3)(1)             & \f(4)(2)\\
(\f(2)(1)-\f(3)(1))  & \f(3)(2)           & (\f(4)(3)-\f(6)(3))   & -(x\f(2)(0)-\f(6)(4))
\end{bmatrix}
\]
It seems that the eliminated variable $\f(4)(3)$ occurs in the matrix, but we can take 
$\f(4)(3)-\f(6)(3)$ as independent variable.

We make the divisibility conditions explicit for $\tau =1$ ($N(3)_4)$) and $\tau =2$
($N(6)_5$).
The $\f(i)(j)$ are  polynomials in $x$, of the form
$\f(i)(j)= \ff(i)(j+4\tau )+\ff(i)(j+4\tau -4)x + \dots + \ff(i)(r)x^{k}$ if $j+4\tau =4k+r$ with $1\leq r\leq 4$.
Recall that we removed the variables of lowest
weight in $\f(2)(0)$, $\f(2)(1)$, $\f(3)(1)$ and $\f(4)(2)$.

For $\tau =1$ the matrix becomes
\[
\begin{bmatrix}
 -\ff(2)(3)   & -\ff(3)(4)  &        \ff(3)(5)      &        \ff(4)(6) \\
\ff(2)(5)-\ff(3)(5) & \ff(3)(6)+\ff(3)(2)x &    \ff(4)(7)-\ff(6)(7)+  (\ff(4)(3)-\ff(6)(3))x &  \ff(6)(8)+\ff(6)(4)x
\end{bmatrix}
\]
and the condition that the minors are divisible by $x$ is obviously 
that they  vanish when $x=0$ is substituted. Therefore the base space is given by the 
vanishing of the minors of
\[
\begin{bmatrix}
 -\ff(2)(3)   & -\ff(3)(4) &        \ff(3)(5)      &        \ff(4)(6) \\
\ff(2)(5)-\ff(3)(5) & \ff(3)(6) &    -\ff(6)(4)\ff(2)(3)-\ff(6)(7) &  \ff(6)(8)
\end{bmatrix}
\]
where we substituted the value for $\ff(4)(7)$.
Note that the variables $\ff(6)(4)$, $\ff(6)(3)$ and $\ff(3)(2)$ do not occur 
in the equations. We recover the result that the base space for $N(3)_{4} $
is the Segre cone.

For  $\tau =2$ the last entry of matrix becomes $\ff(6)(12)+\ff(6)(8)x+\ff(6)(4)x^2-\ff(2)(8)x$.
We apply division with remainder by $x^2-\ff(2)(8)$, leading to
$\ff(6)(12)+\ff(6)(4)\ff(2)(8)+(\ff(6)(8)-\ff(2)(8))x$.
Doing the same for other entries we obtain the transpose of the  matrix
\[
\begin{bmatrix}
 -\ff(2)(7)-\ff(2)(3)x   & \ff(2)(9)-\ff(3)(9)+(\ff(2)(5)-\ff(3)(5))x \\
 \ff(2)(8)-\ff(3)(8) -\ff(3)(4) x &     \ff(3)(10)+\ff(3)(2)\ff(2)(8)+\ff(3)(6)x\\
 \ff(3)(9)+\ff(3)(5)x      &      
 \ff(4)(11)-\ff(6)(11)+ (\ff(4)(3)-\ff(6)(3))\ff(2)(8)+(\ff(4)(7)-\ff(6)(7))x\\
 \ff(4)(10)+\ff(4)(6)x &   
  \ff(6)(12)+\ff(6)(4)\ff(2)(8)+(\ff(6)(8)-\ff(2)(8))x
\end{bmatrix}
\]
Making a coordinate transformation and renaming the variables gives a matrix of the form
\[
\begin{bmatrix}
 \a(7)+\a(3)x   & \a(8)+\a(4) x &        \a(9)+\a(5)x      &        \a(10)+\a(6)x \\
\b(9)+\b(5)x & \b(10)+\b(6)x &    \b(11)+ \b(7)x & \b(12)+\b(8)x
\end{bmatrix}
\]
Division with remainder and taking the $x$-coefficient leads to two equations from each minor:
\begin{gather*}
\a(i+4)\b(j+2)+\a(i)\b(j+6)- \a(j+4)\b(i+2)+\a(j)\b(i+6)\\
\a(i+4)\b(j+6)+\a(i)\b(j+2)\ff(2)(8)- \a(j+4)\b(i+6)-\a(j)\b(i+2)\ff(2)(8)
\end{gather*}

For $N(7)_5$ a computation of the versal deformation up to order 3
allows to recognise the base to be $B_2$ also in this case.


\subsection{The cone over a Grassmannian}
The semigroup  $N(6)_{22}=\langle 6, 7, 8, 9, 10\rangle$ is the first
of the second family of curves studied in \cite{Co21}.
The computation can also easily be done with \textsc{Singular}.
The result is that the base space is the cone over the 
Grassmannian $G(2,5)$. In Tables \ref{tab1} and \ref{tab2}  this base space
is denote by $G$. This shows that $\mathcal{M}_{6,1}^{\scriptscriptstyle N(6)_{22}}$ is
rational.  Equations for the base space
can be recognised because they are the
Pfaffians of a skew-symmetric $5\times5$ matrix, which is the
relation matrix between the equations. Again, a computer
computation will in general not lead to a skew matrix,
but one can obtain that form by row and column operations.

The curve  $N(7)_{26}=\langle 6, 8, 9, 10,11\rangle$ is also 
Gorenstein and has as base space a cone over  $G(2,5)$.
While  $N(6)_{21}=\langle 6, 7, 8, 9, 11\rangle$, which deforms into
$N(6)_{22}$, is not Gorenstein,
but has type $t=2$, the dimension of 
$T^2$ is also five, and a computation with \textsc{Singular} shows that the base
space has the same structure: it is a cone over the 
Grassmannian.

\subsection{A  codimension four base space}
For $N(6)_{20}=\langle 6, 7, 8, 10, 11\rangle$  ($t=3$) and 
$N(7)_{25}=\langle 6, 8,9,10,13\rangle$ ($t=2$) one has $\dim T^2=9$.
We compute the base spaces with Hauser's algorithm. 
It turns out that they have the same structure, called $G'$ in the tables.
 We give here the
details for the first curve.
An additive basis over $\k[x]$ of the coordinate ring is $(1,\y(7),\y(8),\y(10),\y(11), \y(7)\y(8))$. 
We take the following unfolding of the  generators of the ideal:
{\small
\begin{multline*}
\y(7)^2-\y(8)x+(\ff(14)(1)x+\ff(14)(7))\y(7)+\ff(14)(14)+\ff(14)(3)\y(11)+\ff(14)(4)\y(10)+\ff(14)(6)\y(8)\\
   \shoveleft{ \y(8)^2-\y(10)x+(\ff(16)(2)x+\ff(16)(8))\y(8)+(\ff(16)(3)x+\ff(16)(9))\y(7)}
   \\
  \shoveright{ + \ff(16)(16)+\ff(16)(5)\y(11)+\ff(16)(6)\y(10)}\\
   \shoveleft{ \y(7)\y(10)-\y(11)x+(\ff(17)(1)x+\ff(17)(7))\y(10)+(\ff(17)(3)x+\ff(17)(9))\y(8)}\\
  \shoveright{  +(\ff(17)(4)x+\ff(17)(10))\y(7)+\ff(17)(17)+\ff(17)(6)\y(11)}\\
  \shoveleft{  \y(11)\y(7)-x^3+(\ff(18)(2)x+\ff(18)(8))\y(10)+(\ff(18)(4)x+\ff(18)(10))\y(8)+\ff(18)(18)}\\
 \shoveleft{  \y(8)\y(10)-x^3+(\gg(18)(1)x+\gg(18)(7))\y(11)+(\gg(18)(5)x+\gg(18)(11))\y(7)+\gg(18)(18)}\\
   \shoveleft{ \y(11)\y(8)-\y(7)x^2+\ff(19)(19)+(\ff(19)(2)x+\ff(19)(8))\y(11)+(\ff(19)(3)x+\ff(19)(9))\y(10)}\\
 \shoveright{  + (\ff(19)(5)x+\ff(19)(11))\y(8)+(\ff(19)(6)x+\ff(19)(12))\y(7)}\\
   \shoveleft{ \y(10)^2-\y(8)x^2+(\ff(20)(1)x^2+\ff(20)(7)x+\ff(20)(13))\y(7)+\ff(20)(20)+(\ff(20)(3)x+\ff(20)(9))\y(11)} \\
                          \shoveright{+ (\ff(20)(4)x+\ff(20)(10))\y(10)+\ff(20)(5)\y(7)\y(8)+(\ff(20)(6)x+\ff(20)(12))\y(8)}\\
 \shoveleft{ \y(11)\y(10)-\y(8)\y(7)x+(\ff(21)(1)x^2+\ff(21)(7)x+\ff(21)(13))\y(8)
 +(\ff(21)(5)x+\ff(21)(11))\y(10)}\\ 
 \shoveright{+(\ff(21)(2)x^2+\ff(21)(8)x+ \ff(21)(14))\y(7)+
  (\ff(21)(4)x+\ff(21)(10))\y(11)+\ff(21)(21)}\\
 \shoveleft{ \y(11)^2-\y(10)x^2+\ff(22)(7)\y(7)\y(8)+(\ff(22)(6)x+\ff(22)(12))\y(10)+(\ff(22)(5)x+\ff(22)(11))\y(11)}\\
   +(\ff(22)(3)x^2+\ff(22)(9)x+\ff(22)(15))\y(7)+(\ff(22)(2)x^2+\ff(22)(8)x+\ff(22)(14))\y(8)
         +\ff(22)(22)
\end{multline*}}

\noindent 
This shows the variables involved,
except that the $\ff(i)(i)$ and $\gg(18)(18)$ are polynomials in $x$.
In practice it is easier to first work with the coefficients of the $\y(i)$ as polynomials.
On this level  the variables $\ff(i)(i)$ and $\gg(18)(18)$ can be eliminated. After that step the coefficients of $x$ can be taken.
Most variables can be eliminated. What is left are nine rather long polynomials with in total 134 monomials, 
but on closer inspection a
coordinate transformation can be found, leading to the following generators of the ideal
of the base space:
\begin{gather*}
\ff(18)(8)\ff(20)(5)-\ff(17)(6)\ff(22)(7)\\
\ff(14)(4)\ff(20)(9)-\ff(19)(3)\ff(21)(10)+\gg(18)(7)\ff(22)(6)\\
-\ff(16)(6)\ff(19)(3)\ff(20)(5)+\ff(14)(4)\ff(20)(5)\ff(21)(5)-\gg(18)(7)\ff(22)(7)\\
\ff(18)(8)\gg(18)(7)+\ff(16)(6)\ff(17)(6)\ff(19)(3)-\ff(14)(4)\ff(17)(6)\ff(21)(5)\\
-\ff(16)(8)\gg(18)(7)-\ff(16)(6)\ff(20)(9)+\ff(21)(10)\ff(21)(5)\\
-\ff(16)(8)\ff(19)(3)\ff(20)(5)+\ff(20)(9)\ff(22)(7)+\ff(20)(5)\ff(21)(5)\ff(22)(6)\\
\ff(14)(4)\ff(16)(8)\ff(20)(5)-\ff(21)(10)\ff(22)(7)-\ff(16)(6)\ff(20)(5)\ff(22)(6)\\
-\ff(16)(8)\ff(17)(6)\ff(19)(3)+\ff(18)(8)\ff(20)(9)+\ff(17)(6)\ff(21)(5)\ff(22)(6)\\
\ff(14)(4)\ff(16)(8)\ff(17)(6)-\ff(18)(8)\ff(21)(10)-\ff(16)(6)\ff(17)(6)\ff(22)(6)
\end{gather*}
The singular locus consists of two components,
the $(\ff(20)(5),\ff(17)(6))$-plane and the Segre cone
\[
\begin{bmatrix}
 \ff(21)(5)& \ff(16)(6)&   \ff(16)(8) \\
 \ff(19)(3) &\ff(14)(4)&   \ff(22)(6) 
\end{bmatrix}
\]
the other variables being zero.
If $\ff(20)(5)=1$ then $\ff(18)(8)=\ff(17)(6)\ff(22)(7)$
and the generators reduce to the Pfaffians of the matrix
\[
\begin{bmatrix}
0&        \ff(19)(3)& \ff(14)(4)&  \ff(22)(6)& \ff(22)(7) \\ 
-\ff(19)(3)&0&        -\gg(18)(7)& \ff(20)(9)& -\ff(21)(5)\\
-\ff(14)(4)&\gg(18)(7)& 0&         \ff(21)(10)&-\ff(16)(6)\\
-\ff(22)(6)&-\ff(20)(9)&-\ff(21)(10)&0&        -\ff(16)(8)\\
-\ff(22)(7)&\ff(21)(5)& \ff(16)(6)&  \ff(16)(8)& 0         
\end{bmatrix}
\]
in accordance with the fact that the curve deforms into 
$N(6)_{21}$ and $N(6)_{22}$.

On the other component of the singular locus we take 
$\ff(19)(3)=1$ and find $\ff(16)(6)=\ff(14)(4)\ff(21)(5)$, $\ff(16)(8)=\ff(21)(5)\ff(22)(6)$
while the generators reduce to the minors of
\[
\begin{bmatrix}
\ff(22)(7) & \ff(18)(8)  &\ff(16)(6)-\ff(14)(4)\ff(21)(5) &\ff(16)(8)-\ff(21)(5)\ff(22)(6)\\
\ff(20)(5) &  \ff(17)(6)  &           \gg(18)(7)        &  \ff(20)(9)          
\end{bmatrix}
\]
\subsection{Codimension 4 and type 4}
For most of the semigroups with 5 generators and type 4 in the list
the associated monomial curve has $\dim T^2=20$.
Only for $N(6)_{19} =\langle 6, 7, 9, 10, 11 \rangle$ and 
$N(7)_{24} =\langle 6, 8, 9,  11, 13 \rangle$ one has $\dim T^2=21$,
while $\dim T^2=26$
for $N(7)_{12} =\langle5,9,11,12,13 \rangle$. 
The first two curves deform into $N(6)_{20}$ respectively 
$N(7)_{25}$, which are curves with base space $G'$.
We have not been able to determine the exact structure of the base
space; in the tables this is marked by a question mark (?). Only for two 
cases ($N(4)_{7}$ and $N(6)_{12}$, marked !) we give here explicit equations.
For $N(4)_7$ Nakano
computed the base computing in characteristic $7$ \cite{Na16}.
The versal deformation of the monomial curve with semigroup
$N(4)_{7} =\langle 5, 6, 7, 8, 9  \rangle$
was computed with the projection method in \cite{Ste93}.
This computation also takes  care of $N(6)_{12}$ by the following result.

%
\begin{lemma} The curves $N(4)_{7}$ and $N(6)_{12}$ 
have $I^2$-equivalent plane projections.
\end{lemma}
\begin{proof}
The projection onto the plane of the first two coordinates
has equation $\y(6)^5-x^6=0$ for $N(4)_{7}$  and  $\y(7)^5-x^7=0$ for
$N(6)_{12}$. The conductor ideal $I$ has in both cases length 6, being the
difference of the $\delta$-invariant of the plane curve and 
$\delta=g$ of the monomial curve. An easy computation gives that 
$I=\mathfrak{m}^3$, and therefore $x^7-x^6 \in I^2$.
\end{proof}

We slightly modify the computation given in \cite{Ste93} by disregarding
all terms in $I^2$. We start by describing the deformation of the matrix 
defining $\Sigma$:
\[
\begin{bmatrix}
y&e_{01}&e_{02}\\
-(x+e_{10})&y+e_{11}&e_{12}\\
-e_{20}&-(x+e_{21})&y+e_{22}\\
-e_{30}&-e_{31}&-x
\end{bmatrix}
\]
We consider only consider  deformations of negative weight, so we deform the
equation of the plane curve in the following way:
{\small
\[
y^5+c_0x^5+c_1x^4y+c_2x^3y^2+c_{3}x^2y^{3}+d_0x^{4}+d_1x^{3}y+d_2x^{2}y^2+
d_3xy^3+d_{4}y^{4}
\]
}

\noindent
With the help of  \textsc{Singular} \cite{DGPS} the deformation equation \eqref{basic}
was solved for all generators of $N$.  The equation holds modulo the ideal $J$ of
 the base space,  described below. We give here the vector $\alpha$ as
 direct result of the computation.

{\small
\begin{itemize}[rightmargin=0.7\parindent]
\item[$\alpha_1=$]
$x^2c_{0}+\half xc_{2}e_{01}+xc_{3}e_{02}+xd_{0}-yc_{3}e_{01}-yc_{1}e_{10}+yc_{2}e_{11}+yc_{3}e_{12}\linebreak -c_{3}e_{02}e_{10} 
    +c_{2}e_{02}e_{20}+c_{3}e_{02}e_{21}-c_{0}e_{01}e_{30}-c_{1}e_{02}e_{30} 
     +c_{0}e_{12}e_{30}\linebreak -e_{10}e_{12}e_{30}-c_{1}e_{01}e_{31}-c_{2}e_{02}e_{31}+c_{0}e_{22}e_{31}+\half d_{2}e_{01}+d_{3}e_{02}$ 
\item[$\alpha_2=$]
$x^2c_{1}+\half xyc_{2}+xc_{3}e_{01}+xc_{1}e_{10}-\half xc_{2}e_{11}+xd_{1}+\half yd_{2} 
     +d_{3}e_{01}\linebreak-e_{01}^2+d_{4}e_{02}+d_{1}e_{10}-\half d_{2}e_{11}$
\item[$\alpha_3=$]
$\half x^2c_{2}+xyc_{3}+xc_{1}e_{20}+\half xc_{2}e_{21}+\half xd_{2}-yc_{3}e_{10}
+yc_{2}e_{20}+yc_{3}e_{21}\linebreak+yd_{3}-c_{3}e_{02}e_{30}+d_{4}e_{01}
+e_{02}e_{10}-e_{01}e_{11} +d_{1}e_{20}+\half d_{2}e_{21}+e_{01}e_{22}$ 
\item[$\alpha_4=$]
$y^2+xc_{1}e_{30}+\half xc_{2}e_{31}+yc_{2}e_{30}+yc_{3}e_{31}+yd_{4}
   +e_{01}e_{10} +e_{02}e_{20}\linebreak+d_{1}e_{30}+\half d_{2}e_{31}-e_{02}e_{31}$ 
\end{itemize}
}

\noindent To describe the base space it is useful to apply a coordinate
transformation, given by:
{\small
\begin{align*}
d_{0}&\mapsto d_{0}+c_{0}e_{21} \\
d_{1}&\mapsto
d_{1}+c_{1}e_{21}+c_{0}e_{31} \\
 d_{2} &\mapsto
d_{2}+c_{1}e_{31} \\
d_{3}&\mapsto d_{3}+e_{01}+c_{3}e_{10}\\
  d_{4} &\mapsto d_{4}+e_{11}
\end{align*}
}

\noindent
In the new coordinates ideal of the base is given by the following 20 generators, which
we write as sum of corresponding minors:
{\tiny
\begin{multline*}
\begin{bmatrix}
 d_{0}+c_{3}e_{02}-c_{0}e_{10}&
d_{1}-c_{1}e_{10}+c_{3}e_{12}-c_{0}e_{20} &  d_{2}-c_{2}e_{10}-c_{0}e_{30}
& d_{3}  &        d_{4}\\
  e_{01}  &  e_{11}  &  e_{10}-e_{21}   &                                     
e_{20}-e_{31} &   e_{30} \end{bmatrix}\\ 
{}+\begin{bmatrix}
 -c_{0}e_{31}& e_{02}-c_{1}e_{31}&-e_{01}+e_{12}-c_{2}e_{31} &
-e_{11}+e_{22}-c_{3}e_{31}  &e_{21} \\
 e_{02}-c_{0}e_{30}&e_{12}-c_{1}e_{30} &       
 e_{22}-c_{2}e_{30} &e_{10}-c_{3}e_{30}  & e_{20} \end{bmatrix}
\end{multline*}
\setlength{\multlinegap}{3pt}
\begin{multline*}
\begin{bmatrix}
 d_{0}  & d_{1} & d_{2}-e_{02}+c_{1}e_{20}-c_{3}e_{22}        
&d_{3}-e_{12}+c_{2}e_{20}+c_{1}e_{30}  &d_{4}-e_{22}+c_{2}e_{30} 
\\  e_{02}& -e_{01}+e_{12}  &-e_{11}+e_{22}  & e_{21} &e_{31}  
\end{bmatrix} \\
{}
+ {1\over c_0}             
\begin{bmatrix}
c_{0}e_{12}-c_{1}e_{02} &c_{0}e_{22}-c_{2}e_{02}       
&c_{0}e_{10}-c_{3}e_{02}  &c_{0}e_{20}& c_{0}e_{30}-e_{02} \\
 c_{1}e_{01}-c_{0}e_{11} &c_{2}e_{01}-c_{0}e_{10}+c_{0}e_{21} 
&c_{3}e_{01}-c_{0}e_{20}+c_{0}e_{31}&-c_{0}e_{30}&e_{01} \end{bmatrix}
\end{multline*}}

\noindent The $ \frac1{c_0}$ in front of the last matrix means that each minor
has to be divided by $c_0$.
The structure of this base space is discussed in\cite{Ste93}.

According to the formula \eqref{ikweq} we obtain the deformation 
for $N(4)_7$ by adding $\Delta_1$ to $\alpha_1$. For $N(6)_{12}$
the terms in $I^2$ are $x^7+b_0x^6+b_1x^5y+b_2x^4y^2$, so we add
the vector $\Delta_1(x+b_0,b_1,b_2,0)^t$ to $\alpha$. 
In particular we find that the codimension of the base space is the same for
 $N(4)_7$ and $N(6)_{12}$.

The curve $N(6)_{11}  =\langle 5, 8, 9, 11, 12 \rangle$ deforms
with the deformation $(t^5,t^8+s t^7,t^9,t^{11},t^{12})$ of the parametrisation
to a curve with semigroup $\langle 5, 7, 9, 11, 13 \rangle$ but not to the
monomial curve $N(6)_{12}$ with this semigroup:  the deformation
$(t^5,t^7+s' t^8,t^9,t^{11},t^{13})$  of $N(6)_{12}$ is non-trivial of positive weight.
We did not compute the base space for $N(6)_{11}$,  but we 
determined the quadratic part  of the equations. It contains
quadrics of rank two, so the base space is definitely more 
complicated.
It is feasible to compute the deformation with
Hauser's algorithm, but the problem is to simplify the 
resulting equations and write them in a systematic way.
Even for $N(4)_7$ it is very hard to see that the 
equations from Hauser's algorithm give the same base space
as the one above from the projection method.

\begin{prop}\label{propdim}
For all semigroups of  genus $g\leq7$ with $5$ generators and type $4$ one has
$\dim M = 2g +2-\dim T^{1,+}$.
\end{prop}
\begin{proof}
If the monomial curve is negatively graded the result
 follows directly from the Rim--Vitulli formula in Theorem \ref{Stevens}.
 For $N(7)_{13}$ and $N(7)_{27}$ a computation of the deformation up
 to order two yields 20 quadratic equations which are among the equations
 for the tangent cone to the base space (and probably give the tangent cone exactly).
 These 20 equations define a projective scheme of dimension 15, which is therefore
 an upper bound for the dimension of $\M$. At the same time Theorem \ref{Stevens}
 gives 15 as lower bound.
 
 For  $N(7)_{24}$ the situation is more complicated. 
 One of the 21 equations starts with cubic terms. We did compute the base space 
 with Hauser's method. It leads to 256 equations in 63 variables, with 6923 monomials
 in total.  These equations are not independent, in fact they can be reduced
 to 59 equations. Thirty eight variables occur linearly and can be eliminated. The 
 result consists of 21 equations in 25 variables, with 24829 monomials in total; the equation
 starting with cubic terms has 3196 monomials. Taking the lowest degree part
 of each equation gives a manageable system with 163 monomials defining
 a scheme of dimension 15.
\end{proof}

\subsection{Higher codimension}
For the remaining curves we did not determine the base space.
For the curves $N(5)_{12}$, $N(6)_{18}$, $N(7)_{22}$ and $N(7)_{23}$
with type 5 the dimension of $T^2$ is 45, for $N(7)_{34}$ it is 46.
The curves $N(6)_{22}$ and $N(7)_{33}$ have type 6 and $\dim T^2=84$,
while for $N(7)_{39}$, the only type 8 curve,   $\dim T^2=140$.
With decreasing type the
dimension of $T^2$ also decreases: for $N(7)_{35}$, $N(7)_{36}$, $N(7)_{37}$
and $N(7)_{38}$ the dimensions are 28, 19, 14, 14 respectively.
All the curves discussed here are negatively graded, except  $N(7)_{23}$.
For this case the base space was computed up to order two, and a
standard basis of the resulting ideal was computed in finite characterstic,
to speed up the computation. The resulting upper bound for the dimension
of $\mathcal{M}_{7,1}^{\scriptscriptstyle N(7)_{23}}$ again coincides with the lower bound of Theorem \ref{Stevens}.
Herewith Proposition \ref{dimprop} is completely establised.

\end{document}